\documentclass{amsart}
\usepackage{amsfonts}
\usepackage{hyperref}

\newtheorem{theorem}{Theorem}
\theoremstyle{plain}

\newtheorem{corollary}{Corollary}

\newtheorem{lemma}{Lemma}

\newtheorem{proposition}{Proposition}
\newtheorem{remark}{Remark}

\numberwithin{equation}{section}
\input{tcilatex}

\begin{document}
\title[Unification and refinements of certain inequalities]{Unification and
refinements of Jordan, Adamovi\'{c}-Mitrinovi\'{c}and and Cusa's inequalities%
}
\author{Zheng-Hang Yang}
\address{Power Supply Service Center, Zhejiang Electric Power Corporation
Research Institute, Hangzhou City, Zhejiang Province, 310009, China}
\email{yzhkm@163.com}
\date{April 8, 2013}
\subjclass[2010]{Primary 26D05, 33B10; Secondary 26D15, 26E60}
\keywords{inequality, trigonometric function, monotonicity, mean}
\thanks{This paper is in final form and no version of it will be submitted
for publication elsewhere.}

\begin{abstract}
In this paper, we find some new sharp bounds for $\left( \sin x\right) /x$,
which unify and refine Jordan, Adamovi\'{c}-Mitrinovi\'{c}and and Cusa's
inequalities. As applications of main results, some new Shafer-Fink type
inequalities for arc sine function and ones for certain bivariate means are
established, and a simpler but more accurate estimate for sine integral is
derived.
\end{abstract}

\maketitle

\section{Introduction}

The classical Jordan's inequality \cite{Mitrinovic.AI.1970} states that for $%
x\in (0,\pi /2)$ 
\begin{equation}
\frac{2}{\pi }<\frac{\sin x}{x}<1.  \label{Jordan}
\end{equation}

Some new developments on refinements, generalizations and applications of
Jordan's inequality can be found in \cite{Qi.JIA.2009.271923} and related
references therein.

In the recent past, the following two-side inequality%
\begin{equation}
\left( \cos x\right) ^{1/3}<\frac{\sin x}{x}<\frac{2+\cos x}{3}\text{ \ }%
\left( 0<x<\frac{\pi }{2}\right)  \label{M-C}
\end{equation}%
has attracted the attention of many scholars (see, e.g., \cite%
{Chen.JIA.2011.136}, \cite{Iyengar.6.1945}, \cite{Mortitc.MIA.14.3.2011}, 
\cite{Neuman.MIA.13.4.2010}, \cite{Neuman.JMI.5.4.2011}, \cite%
{Qi.JIA.2009.271923}, \cite{Sandor.RGMIA.8.3.2005}, \cite{Wu.PMD.75.3-4.2009}%
, \cite{Zhu.CMA.58.2009}, \cite{Yang.arxiv.1206.4911.2012}, \cite%
{Yang.arxiv.1206.5502.2012}, \cite{Yang.JMI.2013}), where the left
inequality was obtained by Adamovi\'{c} and Mitrinovi\'{c} (see \cite[2, p.
238]{Mitrinovic.AI.1970}), while the right one is due to Cusa and Huygens
(see, e.g., \cite{Huygens}) and it is now known as \emph{Cusa's inequality} 
\cite{Chen.JIA.2011.136}, \cite{Mortitc.MIA.14.3.2011}, \cite%
{Neuman.MIA.13.4.2010}, \cite{Sandor.RGMIA.8.3.2005}, \cite{Zhu.CMA.58.2009}.

In \cite{Oppenheim.AMM.64.6.1957}, the following problem was posed: For each 
$p>0$ there is a greatest $q$ and a least $r$ such that%
\begin{equation*}
\frac{q\sin x}{1+p\cos x}<x<\frac{r\sin x}{1+p\cos x}
\end{equation*}%
for $x\in \left( 0,\pi /2\right) $. Determine $q$ and $r$ as functions of $p$%
. It was solved by Carver in \cite{Carver.AMM.65.2.1958}. In \cite[p. 238,
3.4.15]{Mitrinovic.AI.1970}, it was listed that 
\begin{equation*}
\frac{\left( 1+p\right) \sin x}{1+p\cos x}<x<\frac{(\pi /2)\sin x}{1+p\cos x}
\end{equation*}%
for $p\in (0,1/2]$ and $x\in \lbrack 0,\pi /2]$. Wu \cite{Wu.SCM.11.4.2008}
proved that%
\begin{equation}
\frac{\left( 1+p\right) \cos x}{1+p\cos x}<\frac{\sin x}{x}<\frac{1+q}{%
1+q\cos x}  \label{Wu1}
\end{equation}%
hold for $x\in \left( 0,\pi /2\right) $, $p\in \left[ -1,2\right] $, $q\in
\lbrack -1/4,\infty )$. In particular, he obtained that for $x\in \left(
0,\pi /2\right) $%
\begin{equation}
\frac{3\cos x}{1+2\cos x}<\frac{\sin x}{x}<\frac{3}{4-\cos x}.  \label{Wu2}
\end{equation}%
The first inequality in (\ref{Wu2}) is actually equivalent to Huygen's one:%
\begin{equation*}
2\frac{\sin x}{x}+\frac{\tan x}{x}>3.
\end{equation*}%
Jiang \cite{Jiang.CM.23.4.2007} showed that for $x\in \left( 0,\pi /2\right) 
$%
\begin{equation*}
\frac{\sin x}{x}>\frac{1+2\cos x}{2+\cos x}.
\end{equation*}%
Li and He gave an improvement of (\ref{Wu2}) as follows:%
\begin{equation}
\frac{7+5\cos x}{11+\cos x}<\frac{\sin x}{x}<\frac{9+6\cos x}{14+\cos x}.
\label{Li-He}
\end{equation}

The main purpose of this paper is to give sharp bounds for $\left( \sin
x\right) /x$ in terms of the functions $H_{1}\left( \cos t,p\right) $ and $%
H_{2}\left( \cos t,p\right) $, where%
\begin{eqnarray}
H_{1}\left( x,p\right) &=&\frac{2p+\left( p+3\right) x}{3p+1+2x}\text{, }%
x\in \left( 0,1\right) ,\ p\in (-\infty ,-1]\cup \lbrack 0,\infty ),
\label{H_1} \\
H_{2}\left( x,p\right) &=&\frac{3p+1}{\pi p}\frac{2p+\left( p+3\right) x}{%
\left( 3p+1\right) +2x}\text{, }x\in \left( 0,1\right) ,\ p\in (-\infty
,-1]\cup (0,\infty ).  \label{H_2}
\end{eqnarray}%
The rest of this paper is organized as follows. For later use, some lemmas
are given in Section 2. Main results and their proofs are proven in the
third section, in which Theorem \ref{Theorem Ma} unify and generalize Jordan
and Cusa's inequalities, Theorem \ref{Theorem Mc} shows that Adamovi\'{c}%
-Mitrinovi\'{c}and and Cusa's inequalities (\ref{M-C}) can be interpolated
by $H_{1}\left( \cos x,p\right) $ for suitable $p$, Theorem \ref{Theorem Mg}
gives a hyperbolic version of Theorem \ref{Theorem Ma}. In Section 4, as
applications of main results, some new Shafer-Fink type inequalities for arc
sine function and ones for certain bivariate means are established, and a
simpler but more accurate estimate for sine integral is derived.

\section{Lemmas}

\begin{lemma}
\label{Lemma H1-2}Let $H_{1}$ and $H_{2}$ be defined by (\ref{H_1}) and (\ref%
{H_2}) respectively. Then $H_{1}$ and $H_{2}$ are increasing and decreasing
in $p$ on $(-\infty ,-1]\cup (0,\infty )$, respectively, with the limits 
\begin{eqnarray}
\lim_{p\rightarrow -\infty }H_{1}\left( x,p\right) &=&\lim_{p\rightarrow
\infty }H_{1}\left( x,p\right) =\frac{2+x}{3},  \label{Limit 1} \\
\lim_{p\rightarrow -\infty }H_{2}\left( x,p\right) &=&\lim_{p\rightarrow
\infty }H_{2}\left( x,p\right) =\frac{2+x}{\pi }  \label{Limit 2}
\end{eqnarray}
\end{lemma}

\begin{proof}
Partial derivative calculations yield%
\begin{eqnarray*}
\frac{\partial H_{1}}{\partial p} &=&\frac{2\left( x-1\right) ^{2}}{\left(
3p+1+2x\right) ^{2}}>0, \\
\frac{\partial H_{2}}{\partial p} &=&-\frac{3x}{\pi p^{2}\left(
3p+2x+1\right) ^{2}}\left( p+1\right) \left( \left( 5-2x\right) p+\left(
2x+1\right) \right) .
\end{eqnarray*}%
If $p\in \left( 0,\infty \right) $, then it is cleat that $\partial
H_{2}/\partial p<0$. If $p\in $ $\left( -\infty ,-1\right) $, then $\left(
5-2x\right) p+\left( 2x+1\right) <4\left( x-1\right) <0$, and then $\partial
H_{2}/\partial p<0$.

Simple computation gives (\ref{Limit 1}) and (\ref{Limit 2}), which proves
the lemma.
\end{proof}

\begin{lemma}
\label{Lemma u}Let the function $u_{1},u_{2}$ be defined on $\left(
0,1\right) \times (-\infty ,-1]\cup \lbrack 0,\infty )$ by%
\begin{eqnarray}
u_{1}\left( x,p\right) &=&\left( 2p+\left( 3+p\right) x\right) \left(
3p+1+2x\right) ,  \label{u1} \\
u_{2}\left( x,p\right) &=&2\left( p+3\right) x^{3}+8px^{2}+2p\left(
3p+1\right) x+3\left( p+1\right) ^{2},  \label{u2}
\end{eqnarray}%
respectively. Then $u_{1}\left( x,p\right) ,u_{2}\left( x,p\right) >0$.
\end{lemma}

\begin{proof}
It is not difficult to see that $u_{1}\left( x,p\right) ,u_{2}\left(
x,p\right) >0$ for $p\in \lbrack 0,\infty )$. If $p\in (-\infty ,-1]$, then%
\begin{equation*}
2p+\left( 3+p\right) x<2\left( x-1\right) <0\text{ \ and \ }\left(
3p+1+2x\right) <2\left( x-1\right) <0,
\end{equation*}%
and then $u_{1}\left( x,p\right) >0$. It remains to prove that $u_{2}\left(
x,p\right) >0$ for $p\in (-\infty ,-1]$. Differentiation leads to 
\begin{equation*}
\frac{\partial u_{2}}{\partial p}=\left( 12x+6\right) p+\left(
2x^{3}+8x^{2}+2x+6\right) .
\end{equation*}%
Hence, $\partial u_{2}/\partial p<-\left( 12x+6\right) +\left(
2x^{3}+8x^{2}+2x+6\right) =2x\left( x+5\right) \left( x-1\right) <0$, which
implies that $u_{2}$ is decreasing in $p$ on $\left( -\infty ,-1\right) $,
and therefore,%
\begin{equation*}
u_{2}\left( x,p\right) >u_{2}\left( x,-1\right) =4x\left( x-1\right) ^{2}>0.
\end{equation*}%
This completes the proof.
\end{proof}

\begin{lemma}
\label{Lemma u3}Let the function $u_{3}$ be defined on $\left( 0,1\right)
\times (-\infty ,-1]\cup \lbrack 0,\infty )$ by%
\begin{equation}
u_{3}\left( x,p\right) =\left( p+3\right) ^{2}x^{2}+\left( p+3\right) \left(
7p+3\right) x+\left( -3p^{3}+13p^{2}+21p+9\right) .  \label{u3}
\end{equation}%
Then

(i) $u_{3}\left( x,p\right) \geq 0$ for all $x\in \left( 0,1\right) $ if and
only if $p\in (-\infty ,p_{3}]$, where $p_{3}\approx 5.663$ is the unique
root of the equation $u_{3}\left( 0,p\right) =-3p^{3}+13p^{2}+21p+9=0$;

(ii) $u_{3}\left( x,p\right) \leq 0$ if and only if $p\in \lbrack 9,\infty )$%
;

(iii) For every $p\in \left( p_{3},9\right) $, there is a unique $x_{1}\in
\left( 0,1\right) $ such that $u_{3}\left( x,p\right) <0$ for $x\in \left(
0,x_{1}\right) $ and $u_{3}\left( x,p\right) >0$ for $x\in \left(
x_{1},1\right) $.
\end{lemma}

\begin{proof}
In order to prove the desired results, we need to write $u_{3}\left(
x,p\right) $ as%
\begin{equation*}
u_{3}\left( x,p\right) =\left( \left( p+3\right) x+\tfrac{7p+3}{2}\right)
^{2}-3\left( p-\tfrac{9}{4}\right) \left( p+1\right) ^{2}.
\end{equation*}%
And, we have 
\begin{eqnarray*}
u_{3}\left( 0,p\right) &=&-3p^{3}+13p^{2}+21p+9, \\
u_{3}\left( 1,p\right) &=&-3\left( p+1\right) ^{2}\left( p-9\right) .
\end{eqnarray*}

We claim that there is unique $p_{3}\in \left( 5,6\right) $ such that $%
u_{3}\left( 0,p\right) >0$ for $p\in \left( -\infty ,p_{3}\right) $ and $%
u_{3}\left( 0,p\right) <0$ for $p\in \left( p_{3},\infty \right) $.\ Indeed,
we have 
\begin{equation*}
u_{3}^{\prime }\left( 0,p\right) =-9p^{2}+26p+21=-9\left( p-\tfrac{13-\sqrt{%
358}}{9}\right) \left( p-\tfrac{13+\sqrt{358}}{9}\right) ,
\end{equation*}%
which implies that $u_{3}\left( 0,p\right) $ is increasing on $\left( \tfrac{%
13-\sqrt{358}}{9},\tfrac{13+\sqrt{358}}{9}\right) $ and decreasing on $%
\left( -\infty ,\tfrac{13-\sqrt{358}}{9}\right) \cup \left( \tfrac{13+\sqrt{%
358}}{9},\infty \right) $. Since 
\begin{eqnarray*}
u_{3}\left( 0,\tfrac{13-\sqrt{358}}{9}\right) &=&\frac{13952}{243}-\frac{716%
}{243}\sqrt{358}\approx 1.665>0, \\
u_{3}\left( 0,\tfrac{13+\sqrt{358}}{9}\right) &=&\frac{716}{243}\sqrt{358}+%
\frac{13952}{243}>0, \\
u_{3}\left( 0,\infty \right) &=&-\infty ,
\end{eqnarray*}%
there is unique $p_{3}\in \left( \tfrac{13+\sqrt{358}}{9},\infty \right) $
such that $u_{3}\left( 0,p_{3}\right) =0$ and $u_{3}\left( 0,p\right) >0$
for $p\in \left( -\infty ,p_{3}\right) $ and $u_{3}\left( 0,p\right) <0$ for 
$p\in \left( p_{3},\infty \right) $. An easy calculation reveals that $%
p_{3}\approx 5.663$.

(i) Now we prove the necessary and sufficient condition for $u_{3}\left(
x,p\right) \geq 0$ to hold for all $x\in \left( 0,1\right) $. Since $%
u_{3}\left( x,-3\right) =144>0$, we assume that $p\neq -3$. Denote the
minimum point of $u_{3}\left( x,p\right) $ by $x_{0}$. Then $x_{0}=-\left(
7p+3\right) /\left( 2\left( p+3\right) \right) $. And then, due to $\partial
^{2}u_{3}/\partial x^{2}>0$, $u_{3}\left( x,p\right) \geq 0$ for all $x\in
\left( 0,1\right) $ if and only if at least one of the following cases occur:

Case 1: $x_{0}=-\tfrac{7p+3}{2\left( p+3\right) }\geq 1,u_{3}\left(
1,p\right) \geq 0$. It is derived that $p\in \lbrack -1,-3)$.

case 2: $x_{0}=-\tfrac{7p+3}{2\left( p+3\right) }\leq 0,u_{3}\left(
0,p\right) \geq 0$. It implies that $p\in \left( -\infty ,-3,\right) \cup
\lbrack -3/7,p_{3}]$.

case 3: $x_{0}=-\tfrac{7p+3}{2\left( p+3\right) }\in \left( 0,1\right)
,u_{3}\left( x_{0},p\right) =-3\left( p-\tfrac{9}{4}\right) \left(
p+1\right) ^{2}\geq 0$. It yields $p\in \left( -1,-3/7\right) $.

To sum up, $u_{3}\left( x,p\right) \geq 0$ for all $x\in \left( 0,1\right) $
if and only if $p\in (-\infty ,p_{3}]$.

(ii) It is clear that $u_{3}\left( x,p\right) \leq 0$ if and only if $%
u_{3}\left( 0,p\right) \leq 0$ and $u_{3}\left( 1,p\right) \leq 0$. Solving
the inequalities for $p$ leads to $p\geq 9$.

(iii) In the case when $p\in \left( p_{1},9\right) $, it is seen that $%
u_{3}\left( 0,p\right) <0$, $u_{3}\left( 1,p\right) >0$, $x_{0}=-\left(
7p+3\right) /\left( 2\left( p+3\right) \right) <0$. This indicates that
there is a unique $x_{1}\in \left( 0,1\right) $ such that $u_{3}\left(
x,p\right) <0$ for $x\in \left( 0,x_{1}\right) $ and $u_{3}\left( x,p\right)
>0$ for $x\in \left( x_{1},1\right) $.

This completes the proof.
\end{proof}

Now let us consider the sign of function $g$ defined on $\left( 0,\pi
/2\right) \times (-\infty ,-1]\cup \lbrack 0,\infty )$ by 
\begin{eqnarray}
g\left( t,p\right) &=&t-\frac{\left( 2p+\left( p+3\right) \cos t\right)
\left( 3p+1+2\cos t\right) }{2\left( p+3\right) \cos ^{3}t+8p\cos
^{2}t+2p\left( 3p+1\right) \cos t+3\left( p+1\right) ^{2}}\sin t  \notag \\
&=&t-\frac{u_{1}\left( \cos t,p\right) }{u_{2}\left( \cos t,p\right) }\sin t,
\label{g}
\end{eqnarray}%
where $u_{1}\left( x,p\right) $ and $u_{2}\left( x,p\right) $ are defined by
(\ref{u1}) and (\ref{u2}), respectively. We have

\begin{lemma}
\label{Lemma sgng}Let $g$ defined on $(-\infty ,-1]\cup \lbrack 0,\infty
)\times \left( 0,\pi /2\right) $ by (\ref{g}). Then

(i) $g\left( t,p\right) <0$ for all $t\in \left( 0,\pi /2\right) $ if and
only if $p\in (-\infty ,-1]\cup \left[ 9,\infty \right) $;

(ii) $g\left( t,p\right) >0$ for all $t\in \left( 0,\pi /2\right) $ if and
only if $p\in \lbrack 0,p_{1}]$, where 
\begin{equation}
p_{1}=\frac{2\sqrt{6\pi +1}+3\pi -2}{12-3\pi }\approx 6.343;  \label{p1}
\end{equation}

(iii) in the case when $p\in (p_{1},9)$, there is a unique $t_{0}\in \left(
0,\pi /2\right) $ such that $g\left( t,p\right) >0$ for $t\in \left(
0,t_{0}\right) $ and $g\left( t,p\right) <0$ for $t\in \left( t_{0},\pi
/2\right) $.
\end{lemma}

\begin{proof}
We first give two limit relations as follows:%
\begin{eqnarray}
\lim_{t\rightarrow 0^{+}}\frac{g\left( t,p\right) }{t^{5}} &=&-\frac{1}{45}%
\frac{p-9}{p+1}\text{ if }p\neq -1,  \label{g1} \\
g\left( \tfrac{\pi }{2}^{-},p\right) &=&-\tfrac{12-3\pi }{6\left( p+1\right)
^{2}}\left( p-p_{1}\right) \left( p-p_{2}\right) \text{ if }p\neq -1,
\label{g2}
\end{eqnarray}%
where 
\begin{equation*}
p_{1}=\frac{2\sqrt{6\pi +1}+3\pi -2}{12-3\pi }\approx 6.343\text{ \ and \ }%
p_{2}=\frac{2\sqrt{6\pi +1}-3\pi +2}{12-3\pi }<0.
\end{equation*}%
In fact, for $p\neq -1$, expanding $g\left( t,p\right) $ in power series
gives 
\begin{equation*}
g\left( t,p\right) =-\frac{1}{45}\frac{p-9}{p+1}t^{5}+O\left( t^{7}\right) ,
\end{equation*}%
which implies the first relation. Direct computations yields the second one.

Differentiating $g\left( t,p\right) $ for $t$ leads to 
\begin{eqnarray}
\frac{\partial g}{\partial t} &=&1-\frac{u_{1}\left( \cos t,p\right) }{%
u_{2}\left( \cos t,p\right) }\cos t+\left( \sin ^{2}t\right) \frac{d}{dx}%
\frac{u_{1}\left( x,p\right) }{u_{2}\left( x,p\right) }\bigg|_{x=\cos t} 
\notag \\
&=&\frac{4\left( 1-x\right) \left( 1-x^{2}\right) }{u_{2}^{2}\left(
x,p\right) }\times h\left( x,p\right) ,  \label{dg}
\end{eqnarray}%
where 
\begin{equation}
h\left( x,p\right) =\left( x+\frac{3p+1}{2}\right) \times u_{3}\left(
x,p\right) ,  \label{h}
\end{equation}%
$u_{3}\left( x,p\right) $ is defined by (\ref{u3}) and $x=\cos t\in \left(
0,1\right) $.

(i) We now prove that $g\left( t,p\right) \leq 0$ for all $t\in \left( 0,\pi
/2\right) $ if and only if $p\in (-\infty ,-1]\cup \left[ 9,\infty \right) $%
. The necessity easily follows from the inequalities $\lim_{t\rightarrow
0^{+}}t^{-5}g\left( t,p\right) \leq 0$ and $g\left( \pi /2^{-},p\right) \leq
0$ if $p\neq -1$ and $g\left( t,-1\right) =t-\tan t<0$ together with the
relations (\ref{g1}) and (\ref{g2}).

Next we proves the sufficiency. If $p\in \left[ 9,\infty \right) $, then by
Lemma \ref{Lemma u3} $u_{3}\left( x,p\right) \leq 0$, and then $h\left(
x,p\right) \leq 0$. This indicates that $g$ is decreasing in $t$ on $\left(
0,\pi /2\right) $, and therefore, we get $g\left( t,p\right) <g\left(
0^{+},p\right) =0$. If $p\in (-\infty ,-1]$, then $u_{3}\left( x,p\right)
\geq 0$ and $x+\left( 3p+1\right) /2<x-1<0$, which yields $h\left(
x,p\right) \leq 0$. This also yields that $g$ is decreasing in $t$ on $%
\left( 0,\pi /2\right) $, and so $g\left( t,p\right) <0$.

(ii) Similarly, we can prove that $g\left( t,p\right) >0$ for all $t\in
\left( 0,\pi /2\right) $ if and only if $p\in \lbrack 0,p_{1}]$. If $g\left(
t,p\right) >0$ for all $t\in \left( 0,\pi /2\right) $, then we have $%
\lim_{t\rightarrow 0^{+}}t^{-5}g\left( t,p\right) \geq 0$ and $g\left( \pi
/2^{-},p\right) \geq 0$, which together with (\ref{g1}) and (\ref{g2}) and $%
p\in (-\infty ,-1]\cup \lbrack 0,\infty )$ lead to $p\in \lbrack 0,p_{1}]$.
In order to prove the sufficiency, we distinguish two cases:

In the case of $p\in \lbrack 0,p_{3}]$, by Lemma \ref{Lemma u3} we have $%
u_{3}\left( x,p\right) \geq 0$, which implies that $g$ is increasing in $t$
on $\left( 0,\pi /2\right) $, and so, $g\left( t,p\right) >g\left(
0^{+},p\right) =0$.

In the case of $p\in (p_{3},p_{1}]$, from Lemma \ref{Lemma u3} there is a
unique $x_{1}\in \left( 0,1\right) $ such that $u_{3}\left( x,p\right) <0$
for $x\in \left( 0,x_{1}\right) $ and $u_{3}\left( x,p\right) >0$ for $x\in
\left( x_{1},1\right) $. This in conjunction with (\ref{h}) and (\ref{dg})
shows that $g$ is decreasing in $t$ on $\left( \arccos x_{1},\pi /2\right) $
and increasing on $\left( 0,\arccos x_{1}\right) $, and consequently, we have%
\begin{eqnarray*}
g\left( t,p\right) &>&g\left( 0^{+},p\right) =0\text{ for }t\in \left(
0,\arccos x_{1}\right) , \\
g\left( t,p\right) &>&g\left( \tfrac{\pi }{2}^{+},p\right) =-\tfrac{12-3\pi 
}{6\left( p+1\right) ^{2}}\left( p-p_{1}\right) \left( p-p_{2}\right) \geq 0%
\text{ for }t\in \left( \arccos x_{1},\pi /2\right) ,
\end{eqnarray*}%
which proves the sufficiency,

(iii) In the case when $p\in (p_{1},9)$, we have seen that $g$ is decreasing
in $t$ on $\left( \arccos x_{1},\pi /2\right) $ and increasing on $\left(
0,\arccos x_{1}\right) $ and $g\left( t,p\right) >0$ for $t\in \left(
0,\arccos x_{1}\right) $ but 
\begin{equation*}
g\left( \tfrac{\pi }{2}^{-},p\right) =-\tfrac{12-3\pi }{6\left( p+1\right)
^{2}}\left( p-p_{1}\right) \left( p-p_{2}\right) <0.
\end{equation*}%
Thus, there is a unique $t_{0}\in \left( \arccos x_{1},\pi /2\right) $ such
that $g\left( t,p\right) >0$ for $t\in \left( 0,t_{0}\right) $ and $g\left(
t,p\right) <0$ for $t\in \left( t_{0},\pi /2\right) $.

The whole proof is complete.
\end{proof}

We next observe the function $f$ defined on $(-\infty ,-1]\cup \lbrack
0,\infty )\times \left( 0,\pi /2\right) $ by 
\begin{equation}
f\left( t,p\right) =\ln \frac{\sin t}{t}-\ln H_{1}\left( \cos t,p\right)
=\ln \frac{\sin t}{t}-\ln \frac{2p+\left( p+3\right) \cos t}{\left(
3p+1\right) +2\cos t}.  \label{f}
\end{equation}%
Differentiation yields that%
\begin{eqnarray}
\frac{\partial f}{\partial t} &=&\frac{\cos t}{\sin t}-\frac{1}{t}-\frac{%
2\sin t}{3p+1+2\cos t}+\frac{\left( p+3\right) \sin t}{2p+\left( p+3\right)
\cos t}  \notag \\
&=&\tfrac{2\left( p+3\right) \cos ^{3}t+\left( 3p^{2}+14p+3\right) \cos
^{2}t+3\left( p+1\right) ^{2}\left( \sin ^{2}t\right) +2p\left( 3p+1\right)
\cos t}{\left( \sin t\right) \left( 2p+3\cos t+p\cos t\right) \left(
3p+2\cos t+1\right) }-\frac{1}{t}  \label{df} \\
&=&\frac{1}{\sin t}\frac{u_{2}\left( \cos t,p\right) }{u_{1}\left( \cos
t,p\right) }-\frac{1}{t}=\frac{1}{t\sin t}\frac{u_{2}\left( \cos t,p\right) 
}{u_{1}\left( \cos t,p\right) }\times g\left( t,p\right) ,  \notag
\end{eqnarray}%
where $g\left( t,p\right) $ is defined by (\ref{g}). From Lemmas \ref{Lemma
u} and \ref{Lemma sgng} the following assertion is immediate.

\begin{lemma}
\label{Lemma sgndf}Let $f$ be the function defined on $\left( 0,\pi
/2\right) \times (-\infty ,-1]\cup \lbrack 0,\infty )$ by (\ref{f}). Then

(i) $f$ is decreasing in $t$ on $\left( 0,\pi /2\right) $ if and only if $%
p\in (-\infty ,-1]\cup \left[ 9,\infty \right) $;

(ii) $f$ is increasing in $t$ on $\left( 0,\pi /2\right) $ if and only if $%
p\in \lbrack 0,p_{1}]$, where $p_{1}\approx 6.343$ is given by (\ref{p1});

(iii) in the case when $p\in (p_{1},9)$, there is a unique $t_{0}\in \left(
0,\pi /2\right) $ such that $f$ is increasing in $t$ on $\left(
0,t_{0}\right) $ and decreasing on $\left( t_{0},\pi /2\right) $.$%
\allowbreak $
\end{lemma}

Lastly, for later use, we also give the following

\begin{lemma}
\label{Lemma H_1-x}Let $H_{1}$ be defined on $\left( 0,1\right) \times
(-\infty ,-1]\cup \lbrack 0,\infty )$ by (\ref{H_1}). Then $H_{1}\left(
x^{3},p\right) \geq x$ if and only if $p\in (-\infty ,-1]\cup \lbrack
1,\infty )$, and $H_{1}\left( x^{3},p\right) \leq x$ if and only if $p=0$.
\end{lemma}

\begin{proof}
For $p\in \left( -\infty ,\infty \right) $, we define%
\begin{equation*}
u_{4}\left( x,p\right) =2x^{2}+\left( 1-p\right) x-2p=2\left( x-\tfrac{p-1}{4%
}\right) ^{2}-\tfrac{1}{7}\left( 14p+p^{2}+1\right) .
\end{equation*}%
Then $u_{4}\left( x,p\right) \geq 0$ holds for all $x\in \left( 0,1\right) $
if and only if $p\in (-\infty ,0]$.

In fact, $u_{4}\left( x,p\right) \geq 0$ if and only if at least one case of
the following occur:

Case 1: $\frac{p-1}{4}\geq 1$, $u_{4}\left( 1,p\right) =3-3p\geq 0$. It is
impossible.

Case 2: $\frac{p-1}{4}\leq 0$, $u_{4}\left( 0,p\right) =-2p\geq 0$. It
indicates $p\in (-\infty ,0]$.

Case 3: $0<\frac{p-1}{4}<1$, $u_{4}\left( \frac{p-1}{4},p\right) \geq 0$. It
is impossible.

In the same way, we can prove that $u_{4}\left( x,p\right) \leq 0$ holds for
all $x\in \left( 0,1\right) $ if and only if $p\in \lbrack 1,\infty )$.

We now prove that $H_{1}\left( x^{3},p\right) \geq x$ if and only if $p\in
(-\infty ,-1]\cup \lbrack 1,\infty )$. Factoring yields 
\begin{equation*}
H_{1}\left( x^{3},p\right) -x=-2\left( x-1\right) ^{2}\frac{2x^{2}+\left(
1-p\right) x-2p}{3p+2x^{3}+1}=-\left( x-1\right) ^{2}\frac{u_{4}\left(
x,p\right) }{3p+2x^{3}+1}.
\end{equation*}%
If $p\in (-\infty ,-1]$, then $3p+2x^{3}+1<0$, and then, $H_{1}\left(
x^{3},p\right) \geq x$ if and only if $u_{4}\left( x,p\right) \geq 0$, which
is equivalent to $p\in (-\infty ,-1]\cap (-\infty ,0]=(-\infty ,-1]$. If $%
p\in \lbrack 0,\infty )$, then $3p+2x^{3}+1>0$, and then, $H_{1}\left(
x^{3},p\right) \geq x$ if and only if $u_{4}\left( x,p\right) \leq 0$, which
is equivalent to $p\in \lbrack 0,\infty )\cap \lbrack 1,\infty )=[1,\infty )$%
. Consequently, $H_{1}\left( x^{3},p\right) \geq x$ if and only if $p\in
(-\infty ,-1]\cup \lbrack 1,\infty )$.

Next we show that $H_{1}\left( x^{3},p\right) \leq x$ if and only if $p=0$.
In fact, if $p\in (-\infty ,-1]$, then $H_{1}\left( x^{3},p\right) \leq x$
if and only if $u_{4}\left( x,p\right) \leq 0$, which yields $p\in \lbrack
1,\infty )$. It is clearly a contradiction. If $p\in \lbrack 0,\infty )$,
then the statement in question if and only if $u_{4}\left( x,p\right) \geq 0$%
, which leads us to $p\in \lbrack 0,\infty )\cap (-\infty ,0]=\{0\}$. Thus
the proof is complete.
\end{proof}

\section{Main results}

\begin{theorem}
\label{Theorem Ma}Let $p\in (-\infty ,-1]\cup \lbrack 0,\infty )$. Then for $%
t\in \left( 0,\pi /2\right) $%
\begin{equation}
\frac{\sin t}{t}<\frac{2p+\left( p+3\right) \cos t}{\left( 3p+1\right)
+2\cos t}  \label{Main a}
\end{equation}%
holds if and only if $p\in (-\infty ,-1]\cup \left[ 9,\infty \right) $.
Moreover, we have 
\begin{equation}
H_{2}\left( \cos t,p\right) =\lambda _{p}\frac{2p+\left( p+3\right) \cos t}{%
\left( 3p+1\right) +2\cos t}<\frac{\sin t}{t}<\frac{2p+\left( p+3\right)
\cos t}{\left( 3p+1\right) +2\cos t}=H_{1}\left( \cos t,p\right) 
\label{Main a1}
\end{equation}%
for $p\in (-\infty ,-1]\cup \left[ 9,\infty \right) $, where $\lambda
_{p}=\left( 3p+1\right) /\left( \pi p\right) $ is the best possible. And,
the lower and upper bounds in (\ref{Main a1}) are decreasing and increasing
in $p$ on $(-\infty ,-1]\cup \left( 0,\infty \right) $, respectively.
\end{theorem}

\begin{proof}
Clearly, the desired result is equivalent to $f\left( t,p\right) <0$ if and
only if $p\in (-\infty ,-1]\cup \left[ 9,\infty \right) $. To this end, we
give two limit relations. The first one follows by expanding $f\left(
t,p\right) $ in power series for $t$. We have 
\begin{equation*}
f\left( t,p\right) =-\frac{1}{180}\frac{p-9}{p+1}t^{4}+O\left( t^{6}\right) 
\text{ if }p\neq -1,
\end{equation*}%
which yields%
\begin{equation}
\lim_{t\rightarrow 0^{+}}\frac{f\left( t,p\right) }{t^{4}}=-\frac{1}{180}%
\frac{p-9}{p+1}\text{ if }p\neq -1.  \label{Limit f1}
\end{equation}%
The second one is derived by a simple computation, that is,%
\begin{equation}
f\left( \tfrac{\pi }{2}^{-},p\right) =\ln \frac{3p+1}{\pi p}.
\label{Limit f2}
\end{equation}%
Now we prove that $f\left( t,p\right) <0$ for all $t\in \left( 0,\pi
/2\right) $ if and only if $p\in (-\infty ,-1]\cup \left[ 9,\infty \right) $.

The necessity easily follows by solving the simultaneous inequalities%
\begin{eqnarray*}
\lim_{t\rightarrow 0^{+}}\frac{f\left( t,p\right) }{t^{4}} &=&-\frac{1}{180}%
\frac{p-9}{p+1}\leq 0\ \text{if }p\neq -1\text{ and }f\left( t,-1\right)
=\ln \frac{\sin t}{t}<0, \\
f\left( \tfrac{\pi }{2}^{-},p\right) &=&\ln \frac{3p+1}{\pi p}\leq 0,
\end{eqnarray*}%
which implies $p\in (-\infty ,-1]\cup \left[ 9,\infty \right) $.

The sufficiency is due to Lemma \ref{Lemma sgndf}. In fact, If $p\in
(-\infty ,-1]\cup \left[ 9,\infty \right) $, then by Lemma \ref{Lemma sgndf}
we see that $f$ is decreasing in $t$ on $\left( 0,\pi /2\right) $. Hence, $%
f\left( t,p\right) <f\left( 0^{+},p\right) =0$.

Utilizing the monotonicity of $f$ in $t$ on $\left( 0,\pi /2\right) $ gives (%
\ref{Main a1}). And, from Lemma \ref{Lemma H1-2} it is seen that the lower
and upper bounds in (\ref{Main a1}) are decreasing and increasing in $p$ on $%
(-\infty ,-1]\cup \left( 0,\infty \right) $, respectively.

Thus the proof is finished.
\end{proof}

By Theorem \ref{Theorem Ma} and Lemma \ref{Lemma H1-2}, we have the
following interesting chain of inequalities.

\begin{corollary}
\label{Corollary Ma}For $t\in \left( 0,\pi /2\right) $, we have 
\begin{eqnarray*}
\frac{2}{\pi } &=&H_{2}\left( \cos t,-1\right) <...<H_{2}\left( \cos
t,-\infty \right) =\frac{2+\cos t}{\pi }=H_{2}\left( \cos t,\infty \right)
<... \\
&<&H_{2}\left( \cos t,9\right) <\frac{\sin t}{t}<H_{1}\left( \cos t,9\right)
<...<H_{1}\left( \cos t,\infty \right) =\frac{2+\cos t}{3} \\
&=&H_{1}\left( \cos t,-\infty \right) <...H_{1}\left( \cos t,-1\right) =1.
\end{eqnarray*}
\end{corollary}

\begin{remark}
It is clear that our results unify and refine Jordan and Cusa's
inequalities, and show that the first one in (\ref{Li-He}) is sharp. Also,
Theorem \ref{Theorem Ma} contains other known results, for example, taking $%
p=-3$ in \ref{Main a1} we get 
\begin{equation}
\frac{8}{\pi }\frac{1}{4-\cos t}<\frac{\sin t}{t}<3\frac{1}{4-\cos t},
\label{Wu-Yang}
\end{equation}%
which contain (\ref{Wu2}). After a simple transformation, (\ref{Wu-Yang})
can be written as 
\begin{equation}
\frac{8}{\pi }\frac{t}{\sin t}+\cos t<4<3\frac{t}{\sin t}+\cos t,
\label{Neuman}
\end{equation}%
where the second inequality in (\ref{Neuman}) is due to Neuman and S\'{a}%
ndor \cite[(2.12)]{Neuman.MIA.13.4.2010}.
\end{remark}

\begin{theorem}
\label{Theorem Mb}Let $p\in (-\infty ,-1]\cup \left( 0,\infty \right) $.
Then for $t\in \left( 0,\pi /2\right) $%
\begin{equation}
\frac{2p+\left( p+3\right) \cos t}{\left( 3p+1\right) +2\cos t}<\frac{\sin t%
}{t}  \label{Main b}
\end{equation}%
holds if and only if $p\in \lbrack 0,p_{0}]$, where $p_{0}=\left( \pi
-3\right) ^{-1}\approx 7.063$.

Moreover, for $p\in \lbrack 0,p_{1}]$ we have%
\begin{equation}
H_{1}\left( \cos t,p\right) =\frac{2p+\left( p+3\right) \cos t}{\left(
3p+1\right) +2\cos t}<\frac{\sin t}{t}<\lambda _{p}\frac{2p+\left(
p+3\right) \cos t}{\left( 3p+1\right) +2\cos t}=H_{2}\left( \cos t,p\right) ,
\label{Main b1}
\end{equation}%
where $p_{1}\approx 6.343$, $\lambda _{p}=\left( 3p+1\right) /\left( \pi
p\right) $ is the best possible. And, $H_{1}\left( \cos t,p\right) $, $%
H_{2}\left( \cos t,p\right) $ are decreasing and increasing in $p$ on $%
(-\infty ,-1]\cup \left( 0,\infty \right) $, respectively.

For $p\in (p_{1},p_{0}]$ we have%
\begin{equation}
H_{1}\left( \cos t,p\right) =\frac{2p+\left( p+3\right) \cos t}{\left(
3p+1\right) +2\cos t}<\frac{\sin t}{t}<\delta _{p}\frac{2p+\left( p+3\right)
\cos t}{\left( 3p+1\right) +2\cos t}=\delta _{p}H_{1}\left( \cos t,p\right) ,
\label{Main b2}
\end{equation}%
where $\delta _{p}=\frac{\sin t_{0}}{t_{0}}\frac{\left( 3p+1\right) +2\cos
t_{0}}{2p+\left( p+3\right) \cos t_{0}}$ is the best possible and $t_{0}$ is
the unique root of the equation%
\begin{equation}
\frac{\left( 2p+\left( 3+p\right) \cos t\right) \left( 3p+1+2\cos t\right) }{%
2\left( p+3\right) \cos ^{3}t+8p\cos ^{2}t+2p\left( 3p+1\right) \cos
t+3\left( p+1\right) ^{2}}\sin t=t  \label{df=0}
\end{equation}%
on $\left( 0,\pi /2\right) $.
\end{theorem}

\begin{proof}
Since the inequality (\ref{Main b}) is equivalent to $f\left( t,p\right) >0$%
, it suffices to prove that $f\left( t,p\right) >0$ holds for $t\in \left(
0,\pi /2\right) $ if and only if $p\in \lbrack 0,p_{0}]$.

Similarly, solving the simultaneous inequalities $\lim_{t\rightarrow
0}t^{-4}f\left( t,p\right) \geq 0$ and $f\left( \pi /2^{-},p\right) \geq 0$
and noting $p\in (-\infty ,-1]\cup \left( 0,\infty \right) $ yields $p\in
\lbrack 0,p_{0}]$, which is the necessity.

Conversely, the condition $p\in \lbrack 0,p_{0}]$ is also sufficient for $%
f\left( t,p\right) >0$ to be valid. To this end, we divide the proof into
two cases.

Case 1: $p\in \lbrack 0,p_{1}]$. By Lemma \ref{Lemma sgndf} it is seen that $%
f$ is increasing in $t$ on $\left( 0,\pi /2\right) $, which indicates that $%
f\left( t,p\right) >f\left( 0^{+},p\right) =0$.

Case 2: $p\in (p_{1},p_{0}]$. By Lemma \ref{Lemma sgndf} we see that there
is a unique $t_{0}\in \left( 0,\pi /2\right) $ such that $f$ is increasing
in $t$ on $\left( 0,t_{0}\right) $ and decreasing on $\left( t_{0},\pi
/2\right) $. It is acquired that 
\begin{eqnarray*}
f\left( t_{0},p\right) &>&f\left( t,p\right) >f\left( 0^{+},p\right) =0\text{
for }t\in \left( 0,t_{0}\right) , \\
f\left( t_{0},p\right) &>&f\left( t,p\right) >f\left( \pi /2^{-},p\right)
=\ln \frac{3p+1}{\pi p}\geq 0\text{ for }t\in \left( t_{0},\pi /2\right) ,
\end{eqnarray*}%
that is, 
\begin{equation}
f\left( t_{0},p\right) \geq f\left( t,p\right) >0\text{ for }t\in \left(
0,\pi /2\right) ,  \label{Case 2}
\end{equation}%
which prove the sufficiency.

In the first case, application of the monotonicity of $f$ in $t$ on $\left(
0,\pi /2\right) $ leads to (\ref{Main b1}), and $\lambda _{p}=\left(
3p+1\right) /(\pi p)$. In the second case, (\ref{Case 2}) also yields (\ref%
{Main b1}), and 
\begin{equation*}
\delta _{p}=\exp f\left( t_{0},p\right) =\frac{\sin t_{0}}{t_{0}}\frac{%
\left( 3p+1\right) +2\cos t_{0}}{2p+\left( p+3\right) \cos t_{0}}.
\end{equation*}%
Thus we complete the proof.
\end{proof}

Letting $p=p_{0}=\left( \pi -3\right) ^{-1}$ and solving the equation (\ref%
{df=0}) by mathematical computation software, we find that $t_{0}\approx
1.305$ and $\delta _{p_{0}}\approx 1.0015$. Letting $p=p_{1}$ defined by (%
\ref{p1}) yields $\lambda _{p_{1}}=\left( 3p_{1}+1\right) /\left( \pi
p_{1}\right) \approx 1.0051$. By Theorem \ref{Theorem Mb} we get

\begin{corollary}
For $t\in \left( 0,\pi /2\right) $, we have 
\begin{eqnarray*}
\frac{2p_{0}+\left( p_{0}+3\right) \cos t}{\left( 3p_{0}+1\right) +2\cos t}
&<&\frac{\sin t}{t}<\delta _{p_{0}}\frac{2p_{0}+\left( p_{0}+3\right) \cos t%
}{\left( 3p_{0}+1\right) +2\cos t}, \\
\frac{2p_{1}+\left( p_{1}+3\right) \cos t}{\left( 3p_{1}+1\right) +2\cos t}
&<&\frac{\sin t}{t}<\lambda _{p_{1}}\frac{2p_{1}+\left( p_{1}+3\right) \cos t%
}{\left( 3p_{1}+1\right) +2\cos t},
\end{eqnarray*}%
where $\delta _{p_{0}}\approx 1.0015$ and $\lambda _{p_{1}}\approx 1.0051$
are the best possible constants.
\end{corollary}

Letting $x=\cos ^{1/3}t$ in Lemma \ref{Lemma H_1-x} and using Theorems \ref%
{Theorem Ma} and \ref{Theorem Mb}, we obtain a chain of inequalities
interpolated Adamovi\'{c}-Mitrinovi\'{c}and and Cusa's ones (\ref{M-C}) by $%
H_{1}\left( \cos x,p\right) $.

\begin{theorem}
\label{Theorem Mc}For $t\in \left( 0,\pi /2\right) $, the inequalities 
\begin{eqnarray*}
\frac{2p+\left( p+3\right) \cos t}{\left( 3p+1\right) +2\cos t} &<&\cos
^{1/3}t<\frac{2q+\left( q+3\right) \cos t}{\left( 3q+1\right) +2\cos t}<%
\frac{\sin t}{t} \\
&<&\frac{2r+\left( r+3\right) \cos t}{\left( 3r+1\right) +2\cos t}<\frac{%
2+\cos t}{3}<\frac{2s+\left( s+3\right) \cos t}{\left( 3s+1\right) +2\cos t}
\end{eqnarray*}%
hold if and only if $p=0$, $q\in \lbrack 1,p_{0}]$, $r\in \lbrack 9,\infty )$%
, $s\in (-\infty ,-1]$, where $p_{0}=\left( \pi -3\right) ^{-1}$.
\end{theorem}

Using the monotonicity of $f\left( t,p\right) $ in $t$ on $\left( 0,\pi
/4\right) $ given by parts one and two of Lemma \ref{Lemma sgndf}, we see
that%
\begin{equation*}
\ln \left( \tfrac{4}{\pi }\tfrac{3p+\sqrt{2}+1}{\left( 2\sqrt{2}+1\right) p+3%
}\right) =f\left( \frac{\pi }{4},p\right) \leq f\left( \frac{t}{2},p\right)
=\ln \frac{2\sin \frac{t}{2}}{t}-\ln H_{1}\left( \cos \frac{t}{2},p\right)
\leq f\left( 0,p\right) =0
\end{equation*}%
hold for $p\in (-\infty ,-1]\cup \lbrack 9,\infty )$. And then we have 
\begin{equation}
\tfrac{4}{\pi }\tfrac{3p+\sqrt{2}+1}{\left( 2\sqrt{2}+1\right) p+3}%
H_{1}\left( \cos \frac{t}{2},p\right) \cos \frac{t}{2}<\frac{\sin t}{t}%
=H_{1}\left( \cos \frac{t}{2},p\right) \cos \frac{t}{2}.  \label{Main d}
\end{equation}%
It is clear that the right-hand in (\ref{Main d}) is increasing in $p$ on $%
(-\infty ,-1]\cup \lbrack 0,\infty )$, but the monotonicity of left-hand is
to be checked. We define 
\begin{equation*}
H_{3}\left( x,p\right) =\tfrac{4}{\pi }\tfrac{3p+\sqrt{2}+1}{\left( 2\sqrt{2}%
+1\right) p+3}H_{1}\left( x,p\right) ,
\end{equation*}%
where $x=\cos \left( t/2\right) \in \left[ 1/\sqrt{2},1\right] $.
Logarithmic differentiation leads to 
\begin{eqnarray*}
\frac{\partial \ln H_{3}}{\partial p} &=&\tfrac{3}{\left( 3p+\sqrt{2}%
+1\right) }-\tfrac{2\sqrt{2}+1}{\left( p\left( 2\sqrt{2}+1\right) +3\right) }%
-\tfrac{3}{\left( 3p+2x+1\right) }+\tfrac{x+2}{2p+x\left( p+3\right) } \\
&=&-\tfrac{6\left( 2\sqrt{2}+1\right) \left( x-\frac{\sqrt{2}}{2}\right)
\left( \frac{22-9\sqrt{2}}{7}-x\right) }{\left( 3p+\sqrt{2}+1\right) \left(
p\left( 2\sqrt{2}+1\right) +3\right) \left( 3p+2x+1\right) \left( 2p+x\left(
p+3\right) \right) }\left( p+1\right) \left( p-u_{5}\left( x\right) \right) ,
\end{eqnarray*}%
where%
\begin{equation*}
u_{5}\left( x\right) =\tfrac{\left( 5-2\sqrt{2}\right) x-\left( \sqrt{2}%
+2\right) }{\left( 5\sqrt{2}-2\right) -\left( 2\sqrt{2}+1\right) x}.
\end{equation*}%
Since 
\begin{equation*}
u_{5}^{\prime }\left( x\right) =-\frac{12\left( 3-2\sqrt{2}\right) }{\left( 5%
\sqrt{2}-2-\left( 2\sqrt{2}+1\right) x\right) ^{2}}<0,
\end{equation*}%
we have $-1=u_{5}\left( 1\right) <u_{5}\left( x\right) <u_{5}\left( 1/\sqrt{2%
}\right) =-\left( 24\sqrt{2}+5\right) /49\approx -0.795$. Consequently, $%
\partial \left( \ln H_{4}\right) /\partial p<0$ for $p\in (-\infty ,-1]\cup
\lbrack 0,\infty )$.

The result can be stated as a theorem.

\begin{theorem}
\label{Theorem Md}Let $p\in (-\infty ,-1]\cup \lbrack 0,\infty )$. Then for $%
t\in \left( 0,\pi /2\right) $ the inequalities%
\begin{equation}
\sigma _{p}\frac{2p\cos \frac{t}{2}+\left( p+3\right) \cos ^{2}\frac{t}{2}}{%
\left( 3p+1\right) +2\cos \frac{t}{2}}<\frac{\sin t}{t}<\frac{2p\cos \frac{t%
}{2}+\left( p+3\right) \cos ^{2}\frac{t}{2}}{\left( 3p+1\right) +2\cos \frac{%
t}{2}}.  \label{Main D}
\end{equation}%
hold if and only if $p\in (-\infty ,-1]\cup \lbrack 9,\infty )$, where $%
\sigma _{p}=\tfrac{4}{\pi }\tfrac{3p+\sqrt{2}+1}{\left( 2\sqrt{2}+1\right)
p+3}$ is the best constant. And the right-hand and left-hand in (\ref{Main D}%
) are increasing and decreasing, respectively. (\ref{Main D} is reversed if
and only if $p\in \lbrack 0,p_{1}]$, where $p_{1}\approx 6.343$ is defined
by (\ref{p1}).
\end{theorem}

Putting $p=9,\infty ,0,1$ in Theorem \ref{Theorem Md} we have

\begin{corollary}
For $t\in \left( 0,\pi /2\right) $ the following inequalities hold: 
\begin{eqnarray}
\tfrac{41\left( 2\sqrt{2}-25\right) }{7\pi }\frac{2\cos ^{2}\frac{t}{2}%
+3\cos \frac{t}{2}}{\cos \frac{t}{2}+14} &<&\frac{\sin t}{t}<3\frac{2\cos
^{2}\frac{t}{2}+3\cos \frac{t}{2}}{\cos \frac{t}{2}+14},  \label{Main d1} \\
\tfrac{4\left( 2\sqrt{2}-1\right) }{7}\frac{\cos ^{2}\frac{t}{2}+2\cos \frac{%
t}{2}}{\pi } &<&\frac{\sin t}{t}<\frac{\cos ^{2}\frac{t}{2}+2\cos \frac{t}{2}%
}{3},  \label{Main d2} \\
3\frac{\cos ^{2}\frac{t}{2}}{2\cos \frac{t}{2}+1} &<&\frac{\sin t}{t}<\tfrac{%
4\left( \sqrt{2}+1\right) }{\pi }\frac{\cos ^{2}\frac{t}{2}}{2\cos \frac{t}{2%
}+1},  \label{Main d3} \\
\frac{2\cos ^{2}\frac{t}{2}+1}{\cos \frac{t}{2}+2} &<&\frac{\sin t}{t}<%
\tfrac{2\left( 3-\sqrt{2}\right) }{\pi }\frac{2\cos ^{2}\frac{t}{2}+1}{\cos 
\frac{t}{2}+2}.  \label{Main d4}
\end{eqnarray}
\end{corollary}

Further, let $H_{5}$ be defined on $\left[ 1/\sqrt{2},1\right] \times
(-\infty ,-1]\cup \lbrack 0,\infty )$ by 
\begin{equation*}
H_{4}\left( x,p\right) =\frac{H_{1}\left( 2x^{2}-1,p\right) }{xH_{1}\left(
x,p\right) },
\end{equation*}%
where $H_{1}$ is defined by (\ref{H_1}). We can show that the monotonicity
of $H_{4}$ in $x$ for certain fixed $p$. Differentiation again yields%
\begin{equation*}
\frac{\partial \ln H_{4}\left( x,p\right) }{\partial x}=\tfrac{2}{1+3p+2x}-%
\frac{1}{x}-\tfrac{p+3}{2p+\left( p+3\right) x}+\tfrac{4\left( p+3\right) x}{%
\left( p-3\right) +2\left( p+3\right) x^{2}}-\tfrac{8x}{4x^{2}+3p-1}.
\end{equation*}%
It is easy to verify that 
\begin{eqnarray*}
\frac{\partial \ln H_{4}\left( x,9\right) }{\partial x} &=&-2\frac{\left(
x-1\right) ^{2}\left( 594x^{2}+240x^{3}+8x^{4}+910x+273\right) }{x\left(
2x+3\right) \left( x+14\right) \left( 2x^{2}+13\right) \left(
4x^{2}+1\right) }<0, \\
\frac{\partial \ln H_{4}\left( x,\infty \right) }{\partial x} &=&-2\frac{%
\left( 1-x\right) \left( 2x+1\right) }{x\left( x+2\right) \left(
2x^{2}+1\right) }<0, \\
\frac{\partial \ln H_{4}\left( x,1\right) }{\partial x} &=&2\frac{\left(
1-x\right) \left( 2x^{3}+8x^{2}+x+1\right) }{x\left( 2x-1\right) \left(
x+2\right) \left( 2x^{2}+1\right) }>0.
\end{eqnarray*}%
Consequently, we have%
\begin{equation*}
1=\frac{H_{1}\left( 1,p\right) }{H_{1}\left( 1,p\right) }<\frac{H_{1}\left(
2x^{2}-1,p\right) }{xH_{1}\left( x,p\right) }<\frac{H_{1}\left( 0,p\right) }{%
\frac{1}{\sqrt{2}}H_{1}\left( \frac{1}{\sqrt{2}},p\right) }=\frac{4p}{3p+1}%
\frac{3p+1+\sqrt{2}}{\left( 2\sqrt{2}+1\right) p+3}\text{ for }p=9,\infty .
\end{equation*}%
It is reversed for $p=1$. From these we can obtain the following.

\begin{theorem}
\label{Theorem Me}For $t\in \left( 0,\pi /2\right) $ the following
inequalities hold: 
\begin{equation}
\tfrac{28}{9\pi }\tfrac{6\cos t+9}{\cos t+14}<\tfrac{41\left( 2\sqrt{2}%
-25\right) }{7\pi }\tfrac{2\cos ^{2}\frac{t}{2}+3\cos \frac{t}{2}}{\cos 
\frac{t}{2}+14}<\tfrac{\sin t}{t}<\tfrac{6\cos ^{2}\frac{t}{2}+9\cos \frac{t%
}{2}}{\cos \frac{t}{2}+14}<\tfrac{6\cos t+9}{\cos t+14},  \label{Me1}
\end{equation}%
\begin{equation}
\tfrac{2+\cos t}{\pi }<\tfrac{12\left( 2\sqrt{2}-1\right) }{7\pi }\tfrac{%
\cos ^{2}\frac{t}{2}+2\cos \frac{t}{2}}{3}<\tfrac{\sin t}{t}<\tfrac{\cos ^{2}%
\frac{t}{2}+2\cos \frac{t}{2}}{3}<\tfrac{2+\cos t}{3},  \label{Me2}
\end{equation}%
\begin{equation}
\tfrac{2\cos t+1}{\cos t+2}<\tfrac{2\cos ^{2}\frac{t}{2}+1}{\cos \frac{t}{2}%
+2}<\tfrac{\sin t}{t}<\tfrac{2\left( 3-\sqrt{2}\right) }{\pi }\tfrac{2\cos
^{2}\frac{t}{2}+1}{\cos \frac{t}{2}+2}<\tfrac{4}{\pi }\tfrac{2\cos t+1}{\cos
t+2}.  \label{Me3}
\end{equation}
\end{theorem}

Additionally, Lemma \ref{Lemma sgng} implies an optimal two-side inequality.

\begin{theorem}
\label{Theorem Mf}Let $p\in (-\infty ,-1]\cup \lbrack 0,\infty )$. Then for $%
t\in \left( 0,\pi /2\right) $ the two-side inequality 
\begin{equation}
\frac{u_{2}\left( \cos t,p\right) }{u_{1}\left( \cos t,p\right) }<\frac{\sin
t}{t}<\frac{u_{2}\left( \cos t,q\right) }{u_{1}\left( \cos t,q\right) }
\label{Main f}
\end{equation}%
holds if and only if $p\in (-\infty ,-1]\cup \left[ 9,\infty \right) $ and $%
q\in \lbrack 0,p_{1}]$, where $p_{1}\approx 6.343$. And, for $x\in \left(
0,1\right) $, the function $p\mapsto u_{2}\left( \cos t,p\right)
/u_{1}\left( \cos t,p\right) $ is decreasing on $(-\infty ,-1]\cup \lbrack
0,\infty )$.
\end{theorem}

\begin{proof}
Since $u_{1}\left( x,p\right) ,u_{2}\left( x,p\right) >0$ for $p\in (-\infty
,-1]\cup \lbrack 0,\infty )$ and $x\in \left( 0,1\right) $ by Lemmas \ref%
{Lemma u} and $g\left( t,p\right) $ defined by (\ref{g}) can be written as 
\begin{equation*}
g\left( t,p\right) =-t\frac{u_{1}\left( \cos t,p\right) }{u_{2}\left( \cos
t,p\right) }\left( \frac{\sin t}{t}-\frac{u_{2}\left( \cos t,p\right) }{%
u_{1}\left( \cos t,p\right) }\right) ,
\end{equation*}%
it follows from Lemma \ref{Lemma sgng} that (\ref{Main f}) holds if and only
if $p\in (-\infty ,-1]\cup \left[ 9,\infty \right) $ and $q\in \lbrack
0,p_{1}]$. It remains to check the monotonicity of $u_{2}\left( \cos
t,p\right) /u_{1}\left( \cos t,p\right) $ in $p$. Differentiation yields 
\begin{equation*}
\frac{d}{dp}\frac{u_{2}\left( x,p\right) }{u_{1}\left( x,p\right) }=-6\left(
x+1\right) \left( x-1\right) ^{2}\frac{\left( p+1\right) \left( \left(
5+x\right) p+5x+1\right) }{\left( 2p+3x+px\right) ^{2}\left( 3p+2x+1\right)
^{2}},
\end{equation*}%
where $x\in \left( 0,1\right) $. If $p\in \lbrack 0,\infty )$, then the
numerator of the fraction in right-hand above is clearly positive. $\left(
p+1\right) \left( \left( 5+x\right) p+5x+1\right) >0$. If $p\in (-\infty
,-1] $, then $\left( p+1\right) \leq 0$ and $\left( \left( 5+x\right)
p+5x+1\right) \leq 5(x-1)<0$, which yields that the numerator is nonnegative.

This proves the assertion.
\end{proof}

Similarly, we can obtain a hyperbolic version of Theorems \ref{Theorem Ma}
and \ref{Theorem Mb}

\begin{theorem}
\label{Theorem Mg}Let $p\in (-\infty ,-1]\cup \lbrack 0,\infty )$. Then for $%
t\in \left( 0,\infty \right) $%
\begin{equation}
\frac{2+\left( 1+3p\right) \cosh t}{3+p+2p\cosh t}<\frac{\sinh t}{t}
\label{Main g}
\end{equation}%
holds if and only if $p\in (-\infty ,-1]\cup \left[ \frac{1}{9},\infty
\right) $. It is reversed if and only if $p=0$.
\end{theorem}

\begin{proof}
Let $F$ be the function defined on $[-1,\infty )\times \left( 0,\infty
\right) $ by%
\begin{equation}
F\left( t,p\right) =\frac{3+p+2p\cosh t}{2+\left( 1+3p\right) \cosh t}\sinh
t-t.  \label{F}
\end{equation}%
Then the inequalities (\ref{Main g}) is equivalent to $F\left( t,p\right) >0$%
. Expanding in power series yields%
\begin{equation*}
F\left( t,p\right) =\frac{t^{5}}{180}\frac{p-1/9}{p+1}+O\left( t^{7}\right) ,
\end{equation*}%
which implies 
\begin{equation*}
\lim_{t\rightarrow 0}\frac{F\left( t,p\right) }{t^{5}}=\frac{1}{20}\frac{%
p-1/9}{p+1}\text{ if }p\neq -1\text{ and }F\left( t,-1\right) =\sinh t-t>0.
\end{equation*}%
On the other hand, we have%
\begin{equation*}
\lim_{t\rightarrow \infty }\frac{F\left( t,p\right) }{\sinh t}=\frac{2p}{1+3p%
}.
\end{equation*}

Now we prove desired results.

(i) We first prove $F\left( t,p\right) >0$ holds if and only if $p\in
(-\infty ,-1]\cup \left[ \frac{1}{9},\infty \right) $.

Necessity. If $F\left( t,p\right) >0$ for all $t>0$, then we have%
\begin{equation*}
\left\{ 
\begin{array}{l}
\lim_{t\rightarrow 0}\frac{F\left( t,p\right) }{t^{5}}=\frac{1}{20}\frac{%
p-1/9}{p+1}\geq 0\text{ and }F\left( t,-1\right) =\sinh t-t>0, \\ 
\lim_{t\rightarrow \infty }\frac{F\left( t,p\right) }{\sinh t}=\frac{2p}{1+3p%
}\geq 0.%
\end{array}%
\right.
\end{equation*}%
Solving the inequalities yields $p\in (-\infty ,-1]\cup \left[ \frac{1}{9}%
,\infty \right) $.

Sufficiency. We prove the condition $p\in (-\infty ,-1]\cup \left[ \frac{1}{9%
},\infty \right) $ is sufficient for $F\left( t,p\right) >0$ to hold for $%
t\in \left( 0,\infty \right) $. Differentiation gives%
\begin{eqnarray}
\frac{\partial F}{\partial t} &=&\frac{\left( 3p+1+2\cosh t\right) }{%
2p+\left( p+3\right) \cosh t}\cosh t-\frac{3\left( p+1\right) ^{2}\sinh ^{2}t%
}{\left( 2p+\left( p+3\right) \cosh t\right) ^{2}}-1  \notag \\
&=&\left( x-1\right) ^{2}\frac{2p\left( 3p+1\right) x+\left(
3p^{2}+6p-1\right) }{\left( x+3px+2\right) ^{2}},  \label{dF}
\end{eqnarray}%
where $x=\cosh t\in \left( 1,\infty \right) $.

Due to $p\in (-\infty ,-1]\cup \left[ 1/9,\infty \right) $, we see that $%
2p\left( 3p+1\right) >0$, which yields%
\begin{equation*}
2p\left( 3p+1\right) x+\left( 3p^{2}+6p-1\right) >2p\left( 3p+1\right)
+\left( 3p^{2}+6p-1\right) =\left( p+1\right) \left( 9p-1\right) \geq 0.
\end{equation*}

Then $\partial F/\partial t>0$, that is, $F$ is increasing in $t$ on $\left(
0,\infty \right) $. It is obtained that $F\left( t,p\right) >F\left(
0,p\right) =0$, which proves the sufficiency.

(ii) Next we prove the reverse inequality of (\ref{Main g}) holds if and
only if $p=0$. The necessity follows from%
\begin{equation*}
\left\{ 
\begin{array}{l}
\lim_{t\rightarrow 0}\frac{F\left( t,p\right) }{t^{5}}=\frac{1}{20}\frac{%
p-1/9}{p+1}\leq 0, \\ 
\lim_{t\rightarrow \infty }\frac{F\left( t,p\right) }{\sinh t}=\frac{2p}{1+3p%
}\leq 0,%
\end{array}%
\right.
\end{equation*}%
and the assumption $p\in (-\infty ,-1]\cup \lbrack 0,\infty )$. We get $p=0$.

Now we prove $F\left( t,p\right) <0$ when $p=0$. We have 
\begin{equation*}
\frac{\partial F}{\partial t}=-\frac{\left( x-1\right) ^{2}}{\left(
x+3px+2\right) ^{2}}<0,
\end{equation*}%
then $F\left( t,0\right) <F\left( 0,0\right) =0$.

Thus the proof of Theorem \ref{Theorem Mg} is complete.
\end{proof}

Denote by 
\begin{equation*}
H_{5}\left( x,p\right) =\frac{2+\left( 1+3p\right) x}{3+p+2px}.
\end{equation*}%
It is easy to verify that $H_{5}\left( x,p\right) =H_{1}\left(
x,p^{-1}\right) $ for $p\neq 0$. By Lemma \ref{Lemma H1-2}, we see that $%
H_{5}$ is decreasing in $p$ on $(-\infty ,-1]\cup \lbrack 0,\infty )$. Thus,
as a consequence of Theorem \ref{Theorem Md}, we have

\begin{corollary}
We have 
\begin{eqnarray*}
\frac{2+\cosh t}{3} &>&\frac{\sinh t}{t}>H_{5}\left( \cosh t,\tfrac{1}{9}%
\right) >...>H_{5}\left( \cosh t,\infty \right) \\
&=&\frac{3\cosh t}{2\cosh t+1}=H_{5}\left( \cosh t,-\infty \right)
>...H_{5}\left( \cosh t,-1\right) =1.
\end{eqnarray*}
\end{corollary}

Furthermore, note that $H_{5}\left( x,p\right) =H_{1}^{-1}\left(
x^{-1},p\right) $ and by Lemma \ref{Lemma H_1-x} we have

\begin{corollary}
We have 
\begin{equation*}
\frac{2+\cosh t}{3}>\frac{\sinh t}{t}>\cosh ^{1/3}t>\frac{1+2\cosh t}{%
2+\cosh t}>H_{5}\left( \cosh t,p\right) ,
\end{equation*}
\end{corollary}

where $p\in (-\infty ,-1]\cup (1,\infty )$.

\section{Applications}

In this section, we give some applications of our results.

\subsection{Shafer-Fink type inequalities}

In \cite[3, p. 247, 3.4.31]{Mitrinovic.AI.1970}, it was listed that the
inequality%
\begin{equation*}
\arcsin x>\frac{6\left( \sqrt{x+1}-\sqrt{1-x}\right) }{4+\sqrt{x+1}+\sqrt{1-x%
}}>\dfrac{3x}{2+\sqrt{1-x^{2}}}
\end{equation*}%
hold for $x\in \left( 0,1\right) $, which is due to Shafer. Fink \cite%
{Fink.UBPEF.6.1995} proved that the double inequality%
\begin{equation*}
\tfrac{3x}{2+\sqrt{1-x^{2}}}\leq \arcsin x\leq \tfrac{\pi x}{2+\sqrt{1-x^{2}}%
}
\end{equation*}%
is true for $x\in \left[ 0,1\right] $. There has some improvements,
generalizations of Shafer-Fink inequality (see \cite{Zhu.MIA.8.4.2005}, ).
Letting $\sin t=x$ in Theorems \ref{Theorem Ma}--\ref{Theorem Mf} we can
obtain corresponding Shafer-Fink type inequalities, which clearly contain
many known results. For example, \ref{Theorem Ma} can be changed into the
following

\begin{proposition}
For $x\in \left( 0,1\right) $, the two-side inequality%
\begin{equation}
\tfrac{x}{H_{1}\left( \sqrt{1-x^{2}},p\right) }=x\tfrac{\left( 3p+1\right) +2%
\sqrt{1-x^{2}}}{2p+\left( p+3\right) \sqrt{1-x^{2}}}<\arcsin x<\tfrac{\pi p}{%
3p+1}x\tfrac{\left( 3p+1\right) +2\sqrt{1-x^{2}}}{2p+\left( p+3\right) \sqrt{%
1-x^{2}}}=\tfrac{x}{H_{2}\left( \sqrt{1-x^{2}},p\right) }  \label{P1}
\end{equation}%
holds if and only if $p\in (-\infty ,-1]\cup \left[ 9,\infty \right) $,
where $\pi p/\left( 3p+1\right) $ is the best possible. And, the lower and
upper bounds in (\ref{Main a1}) are decreasing and increasing in $p$ on $%
(-\infty ,-1]\cup \left( 0,\infty \right) $, respectively.

(\ref{P1}) is reversed if $p\in \lbrack 0,p_{1}]$, where $p_{1}\approx 6.343$
is defined by (\ref{p1}).
\end{proposition}

Letting $\sin t=x$. Then $\cos \frac{t}{2}=\frac{1}{2}\left( \sqrt{1+x}+%
\sqrt{1-x}\right) $. Theorems \ref{Theorem Md}\ can be restated as follows.

\begin{proposition}
For $x\in \left( 0,1\right) $, the two-side inequality 
\begin{equation}
2\tfrac{\left( 3p+1\right) \left( \sqrt{1+x}-\sqrt{1-x}\right) +2x}{%
4p+\left( p+3\right) \left( \sqrt{1+x}+\sqrt{1-x}\right) }<\arcsin x<\frac{2%
}{\sigma _{p}}\tfrac{\left( 3p+1\right) \left( \sqrt{1+x}-\sqrt{1-x}\right)
+2x}{4p+\left( p+3\right) \left( \sqrt{1+x}+\sqrt{1-x}\right) }  \label{P2}
\end{equation}%
holds if and only if $p\in (-\infty ,-1]\cup \left[ 9,\infty \right) $,
where $\sigma _{p}=\tfrac{4}{\pi }\tfrac{3p+\sqrt{2}+1}{\left( 2\sqrt{2}%
+1\right) p+3}$ is the best constant. And, the lower and upper bounds in (%
\ref{Main a1}) are decreasing and increasing in $p$ on $(-\infty ,-1]\cup
\left( 0,\infty \right) $, respectively.

(\ref{P1}) is reversed if $p\in \lbrack 0,p_{1}]$, where $p_{1}\approx 6.343$
is defined by (\ref{p1}).
\end{proposition}

As another example, Theorem \ref{Theorem Me} can be rewritten as

\begin{proposition}
For $x\in \left( 0,1\right) $, all the following chains of inequalities
hold: 
\begin{eqnarray}
\tfrac{x}{3}\tfrac{\sqrt{1-x^{2}}+14}{2\sqrt{1-x^{2}}+3} &<&\tfrac{1}{3}%
\tfrac{x+14\left( \sqrt{x+1}-\sqrt{1-x}\right) }{3+\sqrt{x+1}+\sqrt{1-x}}%
<\arcsin x  \label{P31} \\
&<&\tfrac{\left( 41\sqrt{2}+25\right) \pi }{782}\tfrac{x+14\left( \sqrt{x+1}-%
\sqrt{1-x}\right) }{3+\sqrt{x+1}+\sqrt{1-x}}<\tfrac{3\pi x}{28}\tfrac{\sqrt{%
1-x^{2}}+14}{2\sqrt{1-x^{2}}+3},  \notag
\end{eqnarray}%
\begin{equation}
\tfrac{3x}{2+\sqrt{1-x^{2}}}<\tfrac{6\left( \sqrt{x+1}-\sqrt{1-x}\right) }{4+%
\sqrt{x+1}+\sqrt{1-x}}<\arcsin x<\tfrac{\left( 1+2\sqrt{2}\right) \pi }{12}%
\tfrac{6\left( \sqrt{x+1}-\sqrt{1-x}\right) }{4+\sqrt{x+1}+\sqrt{1-x}}<%
\tfrac{\pi x}{2+\sqrt{1-x^{2}}},  \label{P32}
\end{equation}%
\begin{equation}
\tfrac{\pi x}{4}\tfrac{\sqrt{1-x^{2}}+2}{2\sqrt{1-x^{2}}+1}<\tfrac{\left( 
\sqrt{2}+3\right) \pi }{14}\tfrac{x+2\left( \sqrt{x+1}-\sqrt{1-x}\right) }{1+%
\sqrt{x+1}+\sqrt{1-x}}<\arcsin x<\tfrac{x+2\left( \sqrt{x+1}-\sqrt{1-x}%
\right) }{1+\sqrt{x+1}+\sqrt{1-x}}<x\tfrac{\sqrt{1-x^{2}}+2}{2\sqrt{1-x^{2}}%
+1}.  \label{P33}
\end{equation}
\end{proposition}

\begin{remark}
Inequalities (\ref{P32}) is due to Zhu \cite{Zhu.MIA.8.4.2005}.
\end{remark}

\subsection{Inequalities for certain means}

For $a,b>0$ with $a\neq b$, the first and second Seiffert means \cite%
{Seiffert.EM.42.1987}, \cite{Seiffert.DW.29.1995}, Nueman-S\'{a}ndor mean 
\cite{Neuman.MP.14.2003} are defined by%
\begin{eqnarray*}
P &=&P\left( a,b\right) =\frac{a-b}{2\arcsin \frac{a-b}{a+b}}, \\
T &=&T\left( a,b\right) =\frac{a-b}{2\arctan \frac{a-b}{a+b}}, \\
NS &=&NS\left( a,b\right) =\frac{a-b}{2\func{arcsinh}\frac{a-b}{a+b}},
\end{eqnarray*}%
respectively. We also denote the logarithmic mean, arithmetic mean,
geometric mean and quadratic mean of $a$ and $b$ by $L$, $A$, $G$ and $Q$.
There has some inequalities for these means, we quote \cite%
{Neuman.MP.14.2003}, \cite{Neuman.JMI.5.4.2011}, \cite{Chu.JIA.2010.146945}, 
\cite{Chu.JIA.2010.436457}, \cite{Chu.JIA.2011}, \cite{Chu.JIA.2013.10}, 
\cite{Costin.IJMMS.2012.430692}, \cite{Witkowski.MIA.2012.inprint}, \cite%
{Yang.arxiv.1206.5494V1}, \cite{Yang.arxiv.1208.0895V1}, \cite%
{Yang.arxiv.1210.6478}. Now we establish some new ones involving these means.

Let $x=\arcsin \frac{b-a}{a+b}$, $\arctan \frac{b-a}{a+b}$. Then $(\sin
x)/x=P/A$, $\cos x=G/A$; $(\sin x)/x=T/Q$, $\cos x=A/Q$. And then Theorems %
\ref{Theorem Ma}--\ref{Theorem Mf} can be stated as equivalent ones
involving means $P$, $A$, $G$ and $T$, $Q$, $(\sin x)/x=T/Q$, $\cos x=A/Q$.
For example, from Theorem \ref{Theorem Ma} and \ref{Theorem Mf} we have

\begin{proposition}
For $a,b>0$ with $a\neq b$, both the two-side inequalities%
\begin{eqnarray}
\frac{2pA+\left( p+3\right) G}{\left( 3p+1\right) A+2G}A &<&P<A\frac{%
2qA+\left( q+3\right) G}{\left( 3q+1\right) A+2G},  \label{P4} \\
\frac{2pQ+\left( p+3\right) A}{\left( 3p+1\right) Q+2A}Q &<&T<Q\frac{%
2qQ+\left( q+3\right) A}{\left( 3q+1\right) Q+2A}  \label{P5}
\end{eqnarray}%
hold if and only if $p\in \lbrack 0,p_{0}]$ and $q\in (-\infty ,-1]\cup %
\left[ 9,\infty \right) $, where $p_{0}=\left( \pi -3\right) ^{-1}\approx
7.063$.
\end{proposition}

Making changes of variables $x=\func{arctanh}\frac{b-a}{a+b}$, $\func{arcsinh%
}\frac{b-a}{a+b}$ yield $(\sinh x)/x=L/G$, $\cosh x=A/G$; $(\sinh x)/x=NS/A$%
, $\cosh x=Q/A$, respectively. And then, Theorem \ref{Theorem Mg} can be
equivalently written as

\begin{proposition}
For $a,b>0$ with $a\neq b$, both the two-side inequalities%
\begin{eqnarray}
\frac{2G+\left( 1+3p\right) A}{\left( 3+p\right) G+2pA}G &<&L,  \label{P6} \\
\frac{2A+\left( 1+3p\right) Q}{\left( 3+p\right) A+2pQ}A &<&NS  \label{P7}
\end{eqnarray}%
hold if and only if $p\in (-\infty ,-1]\cup \left[ \frac{1}{9},\infty
\right) $. They are reversed if and only if $p=0$.
\end{proposition}

\subsection{The estimate for the sine integral.}

For the estimations for the sine integral defined by 
\begin{equation*}
\limfunc{Si}\left( x\right) =\int_{0}^{x}\frac{\sin t}{t}dt,
\end{equation*}%
there has some results (see \cite{Qi.JMT.12.4.1996}, \cite{Wu.AML.9.12.2006}%
, \cite{Wu.TJM.12.2.2008}). By our results we can obtain many estimates for $%
\limfunc{Si}\left( x\right) $. Here we give a simpler but more accurate one.

\begin{proposition}
For $x\in (0,\pi /2]$, we have 
\begin{equation}
\frac{4\sqrt{2}-2}{7\pi }\left( x+\sin x+8\sin \frac{x}{2}\right) <\limfunc{%
Si}\left( x\right) <\frac{1}{6}\left( x+\sin x+8\sin \frac{x}{2}\right) .
\label{P8}
\end{equation}
\end{proposition}

\begin{proof}
By (\ref{Main d2}) we see that the inequalities%
\begin{equation*}
\tfrac{4\left( 2\sqrt{2}-1\right) }{7}\frac{\cos ^{2}\frac{t}{2}+2\cos \frac{%
t}{2}}{\pi }<\frac{\sin t}{t}<\frac{\cos ^{2}\frac{t}{2}+2\cos \frac{t}{2}}{3%
}
\end{equation*}%
hold for $t\in \left[ 0,\pi /2\right] $. Integrating both sides over $\left[
0,x\right] $ and simple calculation yield (\ref{P8}).
\end{proof}

\begin{remark}
By (\ref{P8}) we have%
\begin{equation*}
1.3682\approx \tfrac{2\sqrt{2}-1}{7\pi }\left( \pi +8\sqrt{2}+2\right)
<\int_{0}^{\pi /2}\frac{\sin t}{t}dt<\frac{1}{12}\left( \pi +8\sqrt{2}%
+2\right) \approx 1.3713.
\end{equation*}
\end{remark}

\end{document}